\documentclass[reqno]{amsart}
\usepackage{mathrsfs}
\usepackage{amssymb}
\usepackage{latexsym}
\usepackage{yfonts}
\usepackage{natbib}

\usepackage{color}

\usepackage{graphicx,color}

\usepackage[plainpages=false,colorlinks,hyperindex,bookmarksopen,linkcolor=red,citecolor=blue,urlcolor=blue]{hyperref}

\bibpunct{[}{]}{;}{n}{,}{,}

\newtheorem{tm}{Theorem}
\newtheorem{defin}{Definition}
\newtheorem{rk}{Remark}
\newtheorem{prop}{Proposition}
\newtheorem{lem}{Lemma}

\numberwithin{equation}{section}

\newcommand{\rr}{{\mathbb R}}

\def\N{{\mathbb N}}

\def\E{{\mathbb E}}
\def\P{{\mathbb P}}

\allowdisplaybreaks

\begin{document}

\title{Coordinate changed random fields on manifolds}

\author{Mirko D'Ovidio}
\address{Department of Basic and Applied Sciences for Engineering, Sapienza University of Rome,  A. Scarpa 00161 Rome, Italy}
\email{mirko.dovidio@uniroma1.it}

\author{Erkan Nane}
\address{Department of Mathematics and Statistics, Auburn University, Auburn, AL 36849, USA}
\email{nane@auburn.edu}

\keywords{Random field, Manifold, Laplace-Beltrami operator, Fractional operator, Karhunen-Loeve expansion, spectral representation}

\date{\today}

\subjclass[2000]{}

\begin{abstract}
We introduce a class of  time dependent random fields on compact Riemannian monifolds. These are represented by  time-changed  Brownian motions.  These processes are  time-changed  diffusion, or the stochastic solution to the equation involving the  Laplace-Beltrami operator and a time-fractional derivative of order $\beta\in (0,1)$.   The time dependent random fields we present in this work can  therefore be  realized through composition and can be viewed as  random fields on randomly varying manifolds.
\end{abstract}

\keywords{ random field on compact manifold, time-changed rotational Brownian motion, stable subordinator, fractional diffusion, sphere, torus.}

\maketitle

\section{Introduction}

In recent years, the study of the random fields on manifolds  attracted the attention of many researchers. They have focused on the construction and characterization of random fields indexed by compact manifolds  such as the sphere $\mathbb{S}^2_r = \{\mathbf{x} \in \mathbb{R}^3:\, |\mathbf{x}|=r\}$, torus, and other compact manifolds:  see, for example, \cite{balMar07, marin06, marin08, marpec08}. In such papers  compact manifolds represent a domain in which the random field is observed. The interest in studying random fields on the sphere is especially represented by the analysis of the Cosmic Microwave Background (CMB) radiation which is currently at the core of physical and cosmological research: see, for example, \cite{Dodelson, kolturner}. CMB radiation is thermal radiation filling the observable universe almost uniformly \cite{PenWil65} and is well explained as radiation associated with an early stage in the development of the universe. For more details on CMB radiation see our recent paper \cite{DovNane1}.

Beside the interest on random fields, the study of fractional diffusion have attracted the attention of many researchers recently.  The fractional diffusions are related to anomalous diffusions or diffusions in non-homogeneous media with random fractal structures; see, for example, \cite{meerschaert-nane-xiao}. Initial study was carried out by  \cite{Koc89, Nig86, Wyss86} in which the authors  established  the mathematical foundations of fractional diffusions: see,  for example,  \cite{NANERW} for a short survey on these results. A large  class of fractional diffusions  are solved by  stochastic  processes that are  time-changed  by inverse stable subordinators: see, for example,  \cite{mnv-09, OB09}.

 Let $(\mathcal{M}, \mu)$ be a smooth connected  Riemannian manifold  of dimension $n\geq 1$ with  Riemannian metric $g$. The associated Laplace-Beltrami operator $\triangle=\triangle_{_\mathcal{M}}$ in $\mathcal{M}$ is  an elliptic, second order, differential operator  defined in the space $C^\infty_0(\mathcal{M})$. In local coordinates, this operator is written as
\begin{equation}
\triangle = \frac{1}{\sqrt{g}}\sum_{i,j=1}^n \frac{\partial}{\partial x_i} \left(  g^{ij} \sqrt{g} \frac{\partial}{\partial x_j}  \right)
\end{equation}
where $\{g_{ij}\}$ is the matrix of the Riemannian metric, $\{g^{ij}\}$ and $g$ are respectively the inverse  and the determinant of $\{g_{ij}\}$.

 For any $y \in \mathcal{M}$, the heat kernel $p(x, y, t)$ is the  fundamental solution to the heat equation
  \begin{equation}
\partial_t u =\triangle u \label{heatM}
\end{equation}
  with initial point source  at $y$. Furthermore,  $p(x,y,t)$ defines an integral kernel of the heat semigroup $P_t = e^{ -t \triangle_\mathcal{M}}$ and  $p(x,y,t)$ is the transition density of a diffusion process on $\mathcal{M}$ which is a Brownian motion generated by $\triangle_\mathcal{M}$. If $\mathcal{M}$ is compact, then $P_t$ is a compact operator on $L^2(\mathcal{M})$.  By the general theory of compact operators, the transition density (heat kernel) $p(x,y,t)$ can be represented as a series expansion in terms of the eigenfunctions of $-\triangle_\mathcal{M}$. The reader is referred to the works by \cite{Chavel84,Davies89,Donnelly-06,Grig89}.

In this paper we consider random fields on the compact Riemannian manifold $\mathcal{M}$, especially 2-manifolds such as torus or the double torus, M\"{o}bious strip, Cylinder and sphere. We will construct a new class of time-dependent random fields indexed by sets of coordinates randomly varying with time  in $\mathcal{M}$. Our construction involves time-changed  Brownian motion on the manifold $\mathcal{M}$ for which we study the corresponding Cauchy problem with  random and deterministic initial conditions.

Let $S^\beta_t$ be a stable subordinator of index $\beta\in (0,1)$ with Laplace transform
 \begin{equation}
 \mathbb{E} \exp ( -s  S^{\beta}_t)  = \exp (- t \, s^\beta). \label{lapH}
\end{equation}  We define by
 \begin{equation}\label{inverse-stable}
 E^{\beta}_t=\inf\{\tau>0:S^{\beta}_\tau>t \}
 \end{equation}
    the inverse of the stable subordinator $S^\beta_t$ of order $\beta \in (0,1)$.
   $E^\beta_t$ has non-negative, non-stationary and non-independent increments (see \cite{MSheff04}).

Our first aim  is to find the unique strong solution to the fractional Cauchy problem
\begin{equation}\label{frac-cauchy-problem}
\partial_t^\beta u(x,t) = \triangle_\mathcal{M} u(x,t),\ \  t>0, x\in \mathcal{M}; \quad u(x,0)=f(x), \  \ x \in \mathcal{M}
 \end{equation}
 where $f$ is  a well specified initial value, and   $\beta \in (0,1)$. The stochastic solution of this equation turns out to be a time-changed Brownian motion, in particular, we get
\begin{equation*}
u(m,t) = \mathbb{E}f(B^m_{E^\beta_t}), \quad m \in \mathcal{M}, \; t>0
\end{equation*}
where $B^m_t$ is a Brownian motion started at $m$ and $E^\beta_t$ is an inverse  stable subordinator  related to the time fractional operator $\partial_t^\beta$ in the sense of Dzhrbashyan-Caputo.

We also consider Cauchy problems involving space fractional operators:
\begin{align}
\mathbb{D}^{\Psi}_\mathcal{M} f(x) = & \int_0^\infty \left( P_s f(x) - f(x) \right) \nu(ds), \quad x\in \mathcal{M} \label{D-op-one}
\end{align}
where $f(x)$ is a well defined function on $\mathcal{M}$, $\nu$ is the L\'{e}vy measure (such that $\int (1 \wedge s)\nu(ds)< \infty$) defining the L\'{e}vy symbol $\Psi$, and $P_t =e^{- t \triangle}$ is the heat semigroup in
\begin{equation}
L^2(\mathcal{M}) = L^2(\mathcal{M}, \mu) := \left\lbrace f:\, \int_\mathcal{M} f^2 d\mu < \infty \right\rbrace.
\end{equation}

We study the heat type  Cauchy problem on $\mathcal{M}$
 \begin{equation}\label{space-frac-cauchy-problem}
\partial_t u(x,t) = \mathbb{D}^\Psi_{\mathcal{M}} u(x,t),\ \  t>0, x\in \mathcal{M}; \quad u(x,0)=f(x), \  \ x \in \mathcal{M}
 \end{equation}
 where $f$ is  a well specified initial value. The stochastic solution of this equation turns out to be  a time-changed Brownian motion (this is also called a subordinate Brownian motion), in particular, we get
\begin{equation*}
u(m,t) = \mathbb{E}f(B^m_{S^\Psi_t}), \quad m \in \mathcal{M}, \; t>0
\end{equation*}
where $B^m_t$ is a Brownian motion started at $m$ and $S^\Psi_t$ is a positive, nondecreasing L\'evy  process with Laplace symbol $\Psi$ with
\begin{equation}\label{gen-lap-transform}
\mathbb{E}\exp (- \xi S^\Psi_t) = \exp (-t \Psi(\xi)).
\end{equation}
The operator $\mathbb{D}^\Psi_{\mathcal{M}}$  turns out to be the infinitesimal generator of the semigroup $P_t^\Psi:=\exp(-t\mathbb{D}^\Psi_{\mathcal{M}})$, $t\geq 0$ on $L^2(\mathcal{M})$: see, for example, \cite{appl2009}.

 In summary, Brownian motion time changed by an inverse subordinator yields a stochastic solution to a time fractional cauchy problem, and Brownian motion time changed by a subordinator which is a positive, nondecreasing L\'evy process, yields a  heat type Cauchy problem with space fractional operator.


Finally, we study the power spectrum of the random fields that are  composed with time-changed  Brownian motions, and find out different covariance structures. In particular, such covariances show different rates of convergence for the covariance of high frequency components.

\subsection{Notations}

\begin{itemize}
\item $B^m_t$, $t\geq 0$, is the Brownian motion on $\mathcal{M}$ started at $m$;
\item $T(m)$, $m \in \mathcal{M}$, is a Gaussian random field indexed by $\mathcal{M}$;
\item $S_t=S^\Psi_t$, $t\geq 0$, is a subordinator with Laplace exponent $\Psi$;
\item $E_t=E^\beta_t$, $t\geq 0$, is an inverse to a stable subordinator $S^\beta_t$, $t>0$, of order $\beta \in (0,1)$;
\item $\mathfrak{T}^\Psi_t(m) = T(B^m_{S_t})$, $t>0$, $m \in \mathcal{M}$;
\item $\mathfrak{T}^\beta_t(m) = T(B^m_{E_t})$, $t>0$, $m \in \mathcal{M}$, $t>0$, $m \in \mathcal{M}$;
\item $T^\Psi_t(m) = \mathbb{E}[\mathfrak{T}^\Psi_t(m) | \mathfrak{F}_T]$, $t>0$, $m \in \mathcal{M}$;
\item $T^\beta_t(m) =\mathbb{E}[\mathfrak{T}^\beta_t(m) | \mathfrak{F}_T]$, $t>0$, $m \in \mathcal{M}$;
\item $\mathfrak{F}_T$ is the $\sigma$-algebra generated by the random field $T$ on $\mathcal{M}$.
\end{itemize}

\section{Preliminaries}

Let  $(\mathcal{M},d, \mu)$ be a manifold with a metric structure where $(\mathcal{M}, d)$ is a locally separable metric space and $\mu$ is a Radon measure supported on $\mathcal{M}$. Let $L^2(\mathcal{M}, \mu)$ be the space of square integrable real-valued functions on $\mathcal{M}$ with finite norm
$$\| u \|_{\mu} : = \left( \int_{\mathcal{M}} |u(m)|^2 \, \mu(dm) \right)^\frac{1}{2}.$$
We are interested in studying the solutions to
\begin{equation}
\triangle \phi + \lambda \phi =0 \label{eigenvalue-problem}
\end{equation}
and
heat equation \eqref{heatM}
from a probabilistic point of view.

The fundamental solution to  heat equation \eqref{heatM} on $\mathcal{M}$ is a
continuous function $p=p(x,y,t)$ on $\mathcal{M}\times \mathcal{M}\times (0,+\infty)$
with
\begin{equation}
\lim_{t \downarrow 0} p(\cdot, y, t) = \delta_y(\cdot), \quad \lim_{t \downarrow 0} p(x, \cdot, t) = \delta_x (\cdot)
\end{equation}
where $\delta_m$ is the Dirac delta function  for $m \in \mathcal{M}$. Furthermore, $p$ is unique and symmetric in the two space variables. Given a continuous initial datum $u_0=f$ we write
\begin{equation}
u(m, t) = P_t f(m) = \mathbb{E} f(B^m_t) = \int_\mathcal{M} p(m,y,t)f(y)\mu(dy).
\end{equation}
One immediately verifies that $P_t$ satisfies the semigroup property: $P_t P_s = P_{t+s}$. We say that $B^m_t$, $t>0$ is a Brownian motion on $\mathcal{M}$ starting at $m \in \mathcal{M}$, that is a measurable map from the probability space $(\Omega, \mathfrak{F}, P )$ to the measurable space $(\mathcal{M}, \mathcal{B}(\mathcal{M}), \mu)$. Furthermore, $p(x,y,t)$ is the  fundamental solution to the heat equation \eqref{heatM}
with point source  initial condition, and $B^m_t$, $t>0$ is a diffusion with continuous trajectories such that
\begin{equation*}
\P\{ B^m_t \in M \} = \int_M p(m,y,t)\mu(dy)
\end{equation*}
for any Borel set $M \subset \mathcal{M}$.

\subsection{Eigenvalue problems and heat kernels}

We follow the presentation in  Section I.3 in Chavel \cite{Chavel84} for stating the following eigenvalue problems.

{\bf Closed eigenvalue problem:} Let $\mathcal{M}$ be a compact, connected  manifold.
Find all real numbers $\lambda $ for which there exist a nontrivial solution $\phi\in C^2(\mathcal{M})$ to
\eqref{eigenvalue-problem}.

{\bf Dirichlet eigenvalue problem:}  For $\partial \mathcal{M}\neq \emptyset,$ $\bar{\mathcal{M}}$ compact and connected, find all real numbers $\lambda$ for which there exist a nontrivial solution $\phi\in C^2(\mathcal{M})\cap C^0(\bar{\mathcal{M}})$ to \eqref{eigenvalue-problem}, satisfying the boundary condition
$$
\phi=0
$$
on $\partial \mathcal{M}$.

{\bf Neumann eigenvalue problem:} For the boundary $\partial \mathcal{M}\neq \emptyset,$ $\bar{\mathcal{M}}$ compact and connected, find all real numbers $\lambda$ for which there exist a nontrivial solution $\phi\in C^2(\mathcal{M})\cap C^1(\bar{\mathcal{M}})$ to \eqref{eigenvalue-problem}, satisfying the boundary condition
$$
\partial_{\bf n} \phi=0
$$
on $\partial \mathcal{M}$ ($\partial_{\bf n}$ is the outward unit normal vector field on  $\partial \mathcal{M}$).

{\bf Mixed eigenvalue problem:} For $\partial \mathcal{M}\neq \emptyset,$ $\bar{\mathcal{M}}$ compact and connected, $\mathcal{N}$ an open submanifold of $\partial \mathcal{M}$,  find all real numbers $\lambda$ for which there exist a nontrivial solution $\phi\in C^2(\mathcal{M})\cap C^1(\mathcal{M}\cup\mathcal{N})\cap  C^0(\bar{\mathcal{M}})$ to \eqref{eigenvalue-problem}, satisfying the boundary conditions
$$
\phi=0\ \mathrm{on}\  \partial \mathcal{M}, \ \partial_{\bf n}\phi=0 \  \mathrm{on }\ \mathcal{N}.
$$
on $\partial \mathcal{M}$.

\begin{tm}
(\citet[page 8]{Chavel84}) For each one of the  eigenvalue problems, the set of eigenvalues consists of a sequence
$$0 \leq \lambda_1 < \lambda_2 \leq \cdots  \uparrow + \infty,$$ and each associated eigenspace is finite dimensional. Eigenspaces belonging to distinct eigenvalues are orthonormal in $L^2(\mathcal{M})$ and $L^2(\mathcal{M})$ is the direct sum of all eigenspaces. Furthermore, each eigenfunction is $C^\infty$ on $\mathcal{M}$.
\end{tm}

In the closed and Neumann eigenvalue problems we have $\lambda_1=0$ and in the Dirichlet and mixed ($\mathcal{N}\neq \mathcal{M}$) eigenvalue problems we have $\lambda_1>0$.

\begin{tm}\cite{Chavel84}
In the case of closed  eigenvalue problem,
each $\phi_j$ is as smooth as the heat kernel $p$. In particular, $p\in C^\infty$ implies $\phi_j\in C^\infty$ for every $j=1,2,\cdots$.
And in this case
\begin{equation}
p(x,y,t)=\sum_{j=1}^\infty e^{-\lambda_j t}\phi_j(x)\phi_j(y) \label{kernel-closed-e}
\end{equation}
with convergence  absolute, and uniform, for each $t>0$.
\end{tm}

\begin{tm}\cite[page 169]{Chavel84}
Given a  connected manifold $\mathcal{M}$ with  piecewise $C^\infty$ boundary and compact closure, there exists a complete orthonormal basis,
$$\{\varphi_1, \varphi_2,\varphi_3,\cdots\}$$
of $L^2(\mathcal{M})$ consisting of Dirichlet eigenfunctions of $\Delta$, with $\varphi_j$ having eigenvalue $\lambda_j$ satisfying
$$0< \lambda_1 < \lambda_2 \leq \cdots  \uparrow + \infty.$$
In particular, each eigenvalue has finite multiplicity. Each
$$\varphi_j \in C^\infty(\mathcal{M})\cap C^1(\bar{\mathcal{M}}).$$
And in this case the fundamental solution, the heat kernel, is given by
\begin{equation}
p(x,y,t)=\sum_{j=1}^\infty e^{-\lambda_j t}\varphi_j(x)\varphi_j(y)
\end{equation}
with convergence  absolute, and uniform, for each $t>0$.
\end{tm}
 For each of the four eigenvalue problems,  as $k\to\infty$
\begin{equation}\label{eigenvalue-asymptotics}
(\lambda_k)^{n/2}\sim (2\pi)^nk/\omega_nV(\mathcal{M})
\end{equation}
where $\omega_n$ is the volume of the unit disk in $\mathbb{R}^n$.
\begin{tm} \label{thm-toharmonic}
(\citet[page 141]{Chavel84}) For any $f \in L^2(\mathcal{M}, \mu)$, the function $P_t f(m)$ converges uniformly, as $t\uparrow +\infty$, to a harmonic function on $\mathcal{M}$. Since $\mathcal{M}$ is compact, the limit function is constant.
\end{tm}


\subsection{Brownian motion on $\mathcal{M}$}

Let $P_t$ be a strongly continuous semigroup on $L^2(\mathcal{M}, \mu)$ and $\mathcal{A}$ be the infinitesimal generator such that
\begin{equation}
\lim_{t \downarrow 0} \left\| \frac{P_t u - u}{t} -  \mathcal{A} u \right\|_{\mu} = 0
\label{limitE}
\end{equation}
for all $u \in  Dom(\mathcal{A}) = \{u \in L^2(\mathcal{M}, \mu) \; \textrm{such that the limit \eqref{limitE} exists} \}$.

See Emery \cite{emery} for a discussion of processes on manifolds.
We have the following result concerning the operator $\mathcal{A}= \triangle$ on $\mathcal{M}$.

\begin{prop}Let $\mathcal{M}$ be a connected and  compact manifold (without boundary!).
The stochastic solution to the Cauchy problem
\begin{equation}
\left\lbrace \begin{array}{l}
 \partial_t u(m,t) =\triangle u(m,t) , \quad m \in \mathcal{M}, \; t>0\\
 u(m,0)=f(m), \quad m \in \mathcal{M}
\end{array} \right .
\end{equation}
is represented by the Brownian motion $B^m_t$, $t>0$ starting from $m \in \mathcal{M}$ at $t=0$ with
\begin{equation}
\begin{split}
u(m,t)&=P_t f(m) =\mathbb{E}f( B^m_t)= \int_{\mathcal{M}}p(m,y,t)f(y)\mu(dy)\\
&= \sum_{j =1}^\infty e^{-t \lambda_j} \kappa_j\, \, \phi_j(m), \quad m \in \mathcal{M}, \; t>0 \label{law-B-on-M}
\end{split}
\end{equation}
where $P_t=\exp(-t\triangle)$ is the semigroup corresponding to Brownian motion  and
\begin{equation}
\kappa_j = \int_\mathcal{M} f(y) \phi_j(y)\mu(dy).
\end{equation}
\end{prop}



Let $\tau_\mathcal{M}(B^m)=\inf\{t>0:\ B^m_t \notin \mathcal{M}\}$ be the first exit time of Brownian motion from $\mathcal{M}$.
The heat equation with Dirichlet boundary conditions is as follows:
\begin{prop}
Let $\mathcal{M}$ be a connected manifold with piecewise $C^\infty$ boundary, and with compact closure.
The stochastic solution to
\begin{equation}
\left\lbrace \begin{array}{l}
 \partial_t u(m,t) =\triangle u(m,t), \quad m \in \mathcal{M}, \; t>0\\
 u(m,0)=f(m), \quad m\in \mathcal{M}\\
 u(m,t)=0 \quad m \in \partial \mathcal{M}, t>0
\end{array} \right .
\end{equation}
is represented by the Brownian motion $B^m_t$ killed on the boundary, $t>0$ starting from $m\in \mathcal{M}$ at $t=0$ with
\begin{equation}
u(m,t)=P_t f(m) =\E(f(B^m_t)I(t<\tau_{\mathcal{M}}(B^m)))= \sum_{j =1}^\infty e^{-t \lambda_j} \kappa_j\, \, \varphi_j(m), \quad m \in \mathcal{M}, \; t>0 \label{law-B-on-M}
\end{equation}
where $P_t=\exp(-t\triangle)$ and
\begin{equation}
\kappa_j = \int_\mathcal{M} f(y) \varphi_j(y)\mu(dy).
\end{equation}

\end{prop}



\subsection{Inverse stable subordinators and Mittag-Leffler function}
\label{secSSIP}

The Dzhrbashyan-Caputo fractional derivative \cite{Caputo} is defined for $0<\beta <1$ as
\begin{equation}\label{capFracDer}
D^\beta_t g(t)=\frac{1}{\Gamma(1-\beta)}\int_0^t \frac{d g(r)}{d r}\frac{dr}{(t-r)^\beta} .
\end{equation}
Its Laplace transform
\begin{equation}\label{CaputolT}
\int_0^\infty e^{-st} D^\beta_t g(t)\,ds=s^\beta \tilde g(s)-s^{\beta-1} g(0)
\end{equation}
incorporates the initial value in the same way as the first
derivative, and $D^\beta_t g(t)$ becomes the ordinary  first derivative $d g(t) / d t $ for $\beta=1$. When  $u(t,x)$ a function of time and space variables, then we use $\partial^\beta_t u(t,x)$ for the Dzhrbashyan-Caputo  fractional derivative of order $\beta\in(0,1)$,

For a function  $g(t)$ continuous in  $t\geq 0$,  the Riemann-Liouville  fractional derivative of
order $0<\nu<1$ is defined by
\begin{equation}\label{R-LDef}
\mathbb{D}^\beta_t g(t)=\frac{1}{\Gamma(1-\beta)}\frac{d }{d t}
\int_0^t \frac{g(r)}{(t-r)^\beta} dr .
\end{equation}
Its Laplace transform is given by
\begin{equation}\label{CaputoLT}
\int_0^\infty e^{-st} \mathbb{D}^\beta_t g(t)\,ds=s^\beta \tilde g(s).
\end{equation}
If $g(\cdot)$ is absolutely continuous on bounded intervals (e.g., if the derivative exists everywhere and is
integrable) then the Riemann-Liouville and Dzhrbashyan-Caputo derivatives are related by
\begin{equation}\label{captuto-R-L-derivatives}
D^\beta_t g(t)=\mathbb{D}^\beta_tu(x,t)-\frac{
t^{-\beta}g(0)}{\Gamma (1-\beta)} .
\end{equation}
The Riemann-Liouville fractional derivative is more general, as it does not require the first derivative to exist.
It is also possible to adopt the right-hand side of \eqref{captuto-R-L-derivatives} as the definition of the Dzhrbashyan-Caputo
derivative; see, for example, Kochubei \cite{koch3}.

A stable subordinator $S^{\beta}_t$, $t>0$, $\beta \in (0,1)$,  is (see \cite{Btoi96}) a L\'evy process with non-negative, independent and stationary increments with Laplace transform in \eqref{lapH}.


The inverse stable subordinator  $E^{\beta}$ defined in \eqref{inverse-stable} with density, say $l_\beta$, satisfies
\begin{equation}
\P\{ E^{\beta}_t < x \} = \P\{ S^{\beta}_x > t \}. \label{relPHL}
\end{equation}

According to \cite{BM01, Dov4, meerschaert-straka},  $E^{\beta}_t$ represents a stochastic solution to
\begin{equation*}
\left( \mathbb{D}^{\beta}_{t} + \frac{\partial}{\partial x} \right) l_{\beta}(x,t)=0, \quad x > 0\; ,t>0,\, \beta \in (0,1)
\end{equation*}
subject to the initial and boundary conditions
\begin{equation}
\left\lbrace \begin{array}{l} l_{\beta}(x,0) = \delta(x), \quad x>0,\\ l_{\beta}(0,t) = t^{-\beta}/\Gamma(1-\beta), \quad t>0. \end{array} \right .\label{fracinicond}
\end{equation}
Due to the fact that $S^\beta_t$, $t>0$ has non-negative increments, that is non-decreasing paths, we have that $E^\beta_t$ is a hitting time. Furthermore, for $\beta \to 1$ we get that
\begin{equation*}
\lim_{\beta \to 1} S^\beta_t = t = \lim_{\beta \to 1} E^\beta_t
\end{equation*}
almost surely (\cite{Btoi96}) and therefore $t$ is the elementary subordinator.

In what follows  we will write $f
\approx g$ and $f \lesssim g$ to mean that for some { positive
$c_{1}$ and $c_{2}$, $c_{1}\leq f/g \leq c_{2}$ and $f \leq c_{1}
g$, respectively.} We will also write $f(t) \sim g(t)$, as
$t\rightarrow \infty $,  to mean that $f(t) / g(t) \rightarrow 1$,
as $t\rightarrow \infty $.

Let
 \begin{equation}\label{mittag-leffler-function}
  E_{\beta}(z)=\sum_{n=0}^\infty \frac{z^n}{\Gamma(1+n\beta)}
  \end{equation} be  the Mittag-Leffler function. By equation (3.16) in \cite{mnv-09} the Laplace transform of $E^{\beta}_t$ is given by
\begin{equation}
\mathbb{E} \exp (- \lambda E^{\beta}_t)  = E_{\beta}(-\lambda t^\beta).  \label{lapL}
\end{equation}

Next we state some of the properties of the Mittag-Leffler function. Let $\beta\in (0,1].$
 As we can immediately check $E_{\beta}(0)=1$ and (see for example \cite{KST06, pod99} )
\begin{equation}
0 \leq E_{\beta}(-z^\beta) \leq \frac{1}{1+z^\beta} \leq 1, \quad z \in [0, +\infty). \label{Ebound}
\end{equation}
Indeed, we have that
\begin{equation}
E_\beta(-z^\beta) \approx 1- \frac{z^\beta}{\Gamma(\beta +1)}  \approx \exp\left( - \frac{z^\beta}{\Gamma(\beta +1)} \right), \quad 0<z \ll 1,
\end{equation}
and
\begin{equation}\label{mittag-leffler-large-asymptotics}
E_\beta(-z^\beta) \approx \frac{z^{-\beta}}{\Gamma(1-\beta)}  - \frac{z^{-2\beta}}{\Gamma(1-2\beta)}+ \ldots , \quad z \to +\infty.
\end{equation}
Thus the Mittag-Leffler function is a stretched exponential with heavy tails. Furthermore, we have that (see \cite[formula 2.2.53]{KST06})
\begin{equation}
\mathbb{D}_{z}^{\beta}\,  E_{\beta}(\mu z^{\beta}) = \frac{z^{- \beta}}{\Gamma(1 -\beta)} + \mu E_{\beta}(\mu z^\beta), \quad \mu \in \mathbb{C}. \label{DfracE2}
\end{equation}


Using Equation \eqref{captuto-R-L-derivatives} formula \eqref{DfracE2} takes the form
\begin{equation}
{D^{\beta}_{z}}\,  E_{\beta}(\mu z^{\beta}) = \mu E_{\beta}(\mu z^\beta), \quad \mu \in \mathbb{C}, \quad \beta \in (0,1). \label{eigenDcaputo}
\end{equation}
Hence in this case we say that  $E_{\beta}(\mu z^{\beta})$ is the eigenfunction of the  Dzhrbashyan-Caputo derivative operator ${D^{\beta}_{z}}$ with   the corresponding eigenvalue $\mu\in \mathbb{C}$.

\section{Space-time fractional  Cauchy problems}
In this section we study time fractional and space fractional Cauchy problems.

\subsection{Time Fractional Cauchy problems in compact manifolds with boundary}

\begin{rk}
\label{strong-caputo-derivative}

We say that  $\Delta u$ exists in the strong sense if it exists pointwise and is continuous in $\mathcal{M}$.

Similarly, we say that $D^\beta_t f(t)$ exists in the strong sense if it exists pointwise and is continuous for  $t\in [0,\infty)$. One sufficient condition for this is the fact $f$ is a $C^1$ function on $[0, \infty)$ with
 $|f'(t)| \leq c \, t^{\gamma -1}$ for some $\gamma >0$. Then by \eqref{capFracDer}, the Caputo fractional derivative $D^\beta_t f(t)$ of $f$ exists  for every $t>0$ and the derivative is continuous in $t>0$.
\end{rk}

Let $\beta\in (0,1)$, $\mathcal{M}_\infty=(0,\infty )\times \mathcal{M}$  and define

\begin{eqnarray}
\mathcal{H}_\Delta(\mathcal{M}_\infty)&\equiv & \left\{u:\mathcal{M}_\infty\to \rr :\ \
\frac{\partial}{\partial t}u, \frac{\partial^\beta}{\partial t^\beta}u, \Delta u\in C(\mathcal{M}_\infty),\right.\nonumber\\
& &\left. \left|\frac{\partial}{\partial t}u(t,x)\right|\leq g(x)t^{\beta -1}, g\in L^\infty(\mathcal{M}), \ t>0 \right\}.\nonumber
\end{eqnarray}


We will write $u\in C^k(\bar{\mathcal{M}})$ to mean that for each fixed $t>0$, $u(t,\cdot)\in C^k(\bar{ \mathcal{M}})$, and
 $u\in C_b^k(\bar{ \mathcal{M}_\infty})$ to mean that $u\in C^k(\bar{\mathcal{M}_\infty})$ and is bounded.

 \begin{tm}\label{frac-laplace-pde}

 Given a  connected manifold $\mathcal{M}$  with  piecewise $C^\infty$ boundary and  compact closure.
   Let $P_t$ be the semigroup of Brownian motion  $\{B^m_t\}$ in $\mathcal{M}$ killed on the boundary $\partial \mathcal{M}$. Let $E_t=E^\beta_t$ be the process inverse  to a stable subordinator of index $\beta\in (0,1)$ independent of $\{B^m_t\}$.   Let $$f \in Dom ( \Delta )\cap C^1(\bar{\mathcal{M}})\cap C^2(\mathcal{M})$$ for which the eigenfunction expansion of $\Delta f$ with respect to the complete orthonormal basis $\{\varphi_n:\ n\in \N \}$ converges uniformly and absolutely.
Then the unique strong solution of
\begin{eqnarray}
 u & \in &
\mathcal{H}_\Delta(\mathcal{M}_\infty)\cap C_b(\bar {\mathcal{M}_\infty}) \cap C^1(\bar {\mathcal{M}}) \nonumber\\
\partial_t^\beta u(m,t) &=&
\Delta u(m,t);  \  \ m\in \mathcal{M}, \ t>0\label{frac-derivative-compact-manifold}\\
u(m, t)&=&0, \ m\in \partial \mathcal{M}, \ t>0, \nonumber\\
u(m,0)& =& f(m), \ m\in \mathcal{M}.\nonumber
\end{eqnarray}
is given by
\begin{eqnarray}
u(m,t)&=&\E[f(B^m_{E_t})I(E_t<\tau(B^m) )]\nonumber\\
&=& \frac{t}{\beta}\int_{0}^{\infty}P_l f(x)g_{\beta}
(tl^{-1/\beta })l^{-1/\beta -1}dl= \int_{0}^{\infty}P_{(t/l)^\beta}f(x)g_{\beta}
(l)dl. \nonumber\\
&=& \sum_{j=1}^\infty E_\beta(-t^\beta \lambda_j) \varphi_j(x) \int_\mathcal{M} \varphi_j(y)f(y)\mu(dy).\nonumber
\end{eqnarray}
Here $g_\beta$ is the density of $S^\beta(1)$.
\end{tm}
\begin{proof}
The proof follows by  similar arguments as in the proof of Theorem 3.1 in \cite{mnv-09}.
\end{proof}

\begin{rk}
Let $f\in C^{2k}_c(\mathcal{M})$ be a $2k$-times continuously differentiable function of
compact support in $\mathcal{M}$.  If $k>1+3n/4$, then the Equation
(\ref{frac-derivative-compact-manifold}) has a classical (strong) solution. In
particular, if $f\in C^{\infty}_c(\mathcal{M})$, then the solution of
Equation (\ref{frac-derivative-compact-manifold}) is in $C^\infty(\mathcal{M}).$ The results in this remark can be seen in connection with Corollary 3.4 in \cite{mnv-09} and  the bounds on the eigenfunctions and on the gradients of the eigenfunctions:  let $n$ denote the dimension of $\mathcal{M}$, for some uniform constants $c_1,c_2>0$
\begin{equation}
\begin{split}
||\varphi_j||_\infty & \leq c_1\lambda_j^{(n-1)/4}||\varphi_j||_2= c_1\lambda_j^{(n-1)/4}\\
||\nabla \varphi_j||_\infty&  \leq c_2\lambda_j^{(n+1)/4}||\varphi_j||_2= c_2\lambda_j^{(n+1)/4},\quad \lambda_j\geq 1
\end{split}
\end{equation}
see  \cite{Donnelly-06} and \cite{shi-xu}, and the references therein. Intrinsic ultracontractivity of the semigroup $P_t$ proved in Kumura \cite{kumura-03} that bounds each eigenfunction with the first eigenfunction can also be used in the proof of this theorem.

\end{rk}

\subsection{Time Fractional Cauchy  problems in compact manifolds without boundary}

Let
\begin{equation}
H^{s}(\mathcal{M}) = \left\lbrace f \in L^2(\mathcal{M}):\,  \sum_{l=0}^{\infty} (\lambda_l)^{2s} \left( \int_\mathcal{M} \phi_l(y) \, f(y) \mu(dy) \right)^2 < \infty  \right\rbrace. \label{SobolevS}
\end{equation}
\begin{tm}\label{Theorem-closed-fractional}
Let  $\beta\in (0,1)$ and $s>(3+3n)/4$.
Let $\mathcal{M}$ be a connected and  compact manifold (without boundary!).
The unique strong solution to the fractional Cauchy problem
\begin{equation}
\left\lbrace \begin{array}{ll}
\partial^\beta_t u(m,t) = \triangle u(m,t), \quad m\in \mathcal{M}, t>0\\
u(m,0)=f(m),\quad  m\in \mathcal{M}, \quad f \in H^{s}(\mathcal{M})
\end{array} \right .
\label{pdeFracmanifold}
\end{equation}
is given by
\begin{equation}
u(m,t)=\mathbb{E} f(B^m_{E_t}) = \sum_{j =1}^\infty E_\beta(-t^\beta \lambda_j) \phi_j(m) \int_\mathcal{M} {\phi_j(y)f(y) \mu(dy)} \label{sol-closed-fractional}
\end{equation}
where $B^m_t$ is a Brownian motion in $\mathcal{M}$ and $E_t=E^\beta_t$ is inverse to a stable subordinator with index $0<\beta<1$.
\end{tm}

\begin{proof}
 The proof follows the main steps in the proof of Theorem 3.1 in \cite{mnv-09}.
The proof is based on the method of separation of variables.  Let $u(m,t)=G(t)F(m)$ be a solution of
(\ref{pdeFracmanifold}).
 Then substituting into  (\ref{pdeFracmanifold}), we get
$$
F(m)D_t^\beta G(t)
= G(t)\triangle  F(m).
$$
Divide both sides by $G(t)F(m)$ to obtain
$$
\frac{D^\beta_t G(t)}{G(t)} = \frac{\triangle  F(m)}{F(m)}= -c.
$$
Then we have
\begin{equation}\label{time-pde}
D^\beta_t G(t) = -c \, G(t), \ t>0
\end{equation}
and
\begin{equation}\label{space-pde}
\triangle F(m)=-c \, F(m), \ m\in \mathcal{M}.
\end{equation}
By the discussion above, the eigenvalue problem (\ref{space-pde}) is solved by an infinite
sequence of pairs $\{\lambda_{j}, \phi_j:\, j\in\mathbb{N} \}$ where $\phi_j$ forms a complete orthonormal set in $L^2(\mathcal{M})$.
In particular, the initial function $f$ regarded as an element of $L^2(\mathcal{M})$ can be represented as
\begin{equation}
f(m)=\sum_{l=1}^\infty \kappa_l\phi_l(m).
\end{equation}
where $\kappa_{l}=\int_{\mathcal{M}}f(m) \phi_l(m)\mu(dm)$.
By equation \eqref{eigenDcaputo} we  see that $\kappa_lE_\beta(-\lambda_l(t)^\beta)$
solves \eqref{time-pde}.  Sum these solutions $ \kappa_lE_\beta(-\lambda_l(t)^\beta) \phi_l(m)$ to (\ref{pdeFracmanifold}), to get
\begin{equation}\label{formal-sol-L-1}
u(t,m)=\sum_{l=1}^\infty
\kappa_lE_\beta(-\lambda_l(t)^\beta) \phi_l(m).
\end{equation}
Since $f\in H^{s}(\mathcal{M}) $ we get
$$
|\kappa_l|\leq C\lambda_l^{-s}
$$
for some $C>0$.
The fact that  the series \eqref{formal-sol-L-1} converges absolutely and uniformly follows from
$$
||\phi_l||_\infty\leq C{\lambda_l}^{(n-1)/4}
$$
and the asymptotics in \eqref{eigenvalue-asymptotics}:
$$||u||_\infty\leq \sum_{l=1}^\infty
|\kappa_l| E_\beta(-\lambda_l(t)^\beta) ||\phi_l||_\infty\leq \sum_{l=1}^\infty \lambda_l^{-s}{\lambda_l}^{(n-1)/4}<\infty $$
as  $s>1/4+3n/4$ and  $\lambda_l\sim C_n l^{2/n}$ as $l\to\infty$.

To show that $u\in C^1(\mathcal{M})$ we use the following from Theorem 1 in Shi and Xu \cite{shi-xu}. : For all $\lambda_l\geq 1$
$$
||\nabla \phi_l||_\infty\leq C\sqrt{\lambda_l}||\phi_l||_\infty
$$
where $C$ is constant depending only on $\mathcal{M}$.
This gives
$$
||\nabla u||_\infty\leq \sum_{l=0}^\infty
|\kappa_l| ||\nabla \phi_l||_\infty\leq  \sum_{l=1}^\infty \lambda_l^{-s}{\lambda_l}^{n/2}<\infty
$$
 Since $s>1/4+3n/4$. Hence $u\in C^1(\mathcal{M})$.

 Similarly
 $$
||\Delta  u||_\infty\leq \sum_{l=1}^\infty
|\kappa_l| ||\Delta \phi_l||_\infty\leq  \sum_{l=1}^\infty \lambda_l^{-s}||{\lambda_l}\phi_l||_\infty<\infty
$$
as $s>3/4+3n/4$.
 We next show that $\partial^\beta_t u$ exists pointwise as a continuous function.
Using \cite[Equation (17)]{krageloh}
$$
\left| \frac{ d E_\beta(-\lambda
t^\beta)}{dt}\right|\leq c\frac{\lambda t^{\beta-1}}{1+\lambda t^\beta}\leq c\lambda t^{\beta -1},
$$
 we get
 $$
||\partial_t  u||_\infty\leq \sum_{l=1}^\infty
|\kappa_l| || \partial_t E_\beta(-\lambda_l(t)^\beta)||\phi_l||_\infty\leq ct^{\beta -1} \sum_{l=1}^\infty \lambda_l^{-s}|{\lambda_l}| ||\phi_l||_\infty<\infty
$$
 as $s>3/4+3n/4$.  It follows from Remark \ref{strong-caputo-derivative}, $\partial_t^\beta u$  exists  pointwise as a continuous function and  is defined as a classical function.
 Hence we can apply the Laplacian and Caputo fractional derivative $\partial^\beta_t$ term by term to \eqref{formal-sol-L-1} to show
$$
(\partial^\beta_t-\Delta )u(t,m)=\sum_{l=1}^\infty \kappa_l\bigg[
  \partial^\beta_tE_\beta(-\lambda_l(t)^\beta) \phi_l(m)-E_\beta(-\lambda_l(t)^\beta)\Delta \phi_l(m)\bigg]=0
$$

The rest of the proof follows similar to the proof of Theorem 3.1 in \cite{mnv-09}.

\end{proof}

\subsection{Space fractional operators}
Let $P_s=e^{s\triangle}$ be the semigroup associated with the Laplace operator  $\triangle$ on the manifold $\mathcal{M}$, and let   $\nu(\cdot)$ be the L\'{e}vy  measure of subordinator $S_t=S^\Psi_t$  such that $\nu(-\infty, 0)=0$, $\int_0^\infty (s \wedge 1)\nu(ds)< \infty$ where $a \wedge b = \min \{a, b\}$.  Recall that
\begin{equation}
\Psi(\xi) = \int_0^\infty \left( 1- e^{-s\xi}\right) \nu(ds) \label{symb-subordinator}
\end{equation}
is the so-called Laplace exponent of the corresponding subordinator  $S_t=S^\Psi_t$, $t\geq 0$ with Laplace transform \eqref{gen-lap-transform}.
The standard way to define fractional differential operators is as in equation \eqref{D-op-one} (see Schilling et al. \cite{book-song-et-al}).

Formula \eqref{D-op-one} is a generalized version of the representation
\begin{equation}
-(-\triangle_{\mathcal{M}})^\alpha f(x) = \int_0^\infty \left( P_sf(x) - f(x) \right) \nu(ds)
\end{equation}
 where in this case $\nu(ds)=\alpha s^{-\alpha -1}/\Gamma(1-\alpha)ds$ is the (density of the) L\'{e}vy measure of a stable subordinator.

\begin{rk}
After some calculation, the operator \eqref{D-op-one} can be formally rewritten as
\begin{equation}
\mathbb{D}^{\Psi}_{\mathcal{M}} f(x) = - \int_{\mathcal{M}} f(y)  J(x,y) \mu(dy) \label{fract-op-Psi-1}
\end{equation}
where
\begin{align*}
J(x,y) = \sum_{j \in \mathbf{N}} \Psi(\lambda_j) \, \phi_j(x) \, \phi_j(y).
\end{align*}
%
%
%
%

In order to arrive at \eqref{fract-op-Psi-1}  we observe that, for a suitable function  $f$ on $\mathcal{M}$ for which a series representation by means of the orthonormal system $\{\phi_j \}_{j \in \mathbf{N}}$ holds true, we have that
\begin{equation}
P_sf(x) = \sum_{j \in \mathbf{N}} e^{-s \lambda_j } \phi_j(x) \, f_j \label{semi-group-P}
\end{equation}
is the transition semigroup of a Brownian motion $B_t$, $t\geq 0$, on the manifold $\mathcal{M}$. We recall that $f_j = \int_\mathcal{M} f(x) \phi_j(x)\mu(dx)$. Therefore, the semigroup $u(x,t)=P_tf(x)$ solves the heat equation \eqref{heatM}. From \eqref{semi-group-P}, the operator \eqref{D-op-one} takes the form
\begin{align*}
\mathbb{D}^{\Psi}_\mathcal{M} f(x) = & \int_0^\infty \left( P_s f(x) - P_0f(x) \right) \nu(ds)\\
= & \sum_{j \in \mathbf{N}} f_j \phi_j(x) \int_0^\infty \left( e^{-s \lambda_j} - 1 \right)\nu(ds)\\
= & -\sum_{j \in \mathbf{N}} f_j \phi_j(x) \Psi(\lambda_j)
\end{align*}
and therefore, we formally get that
\begin{align*}
\mathbb{D}^{\Psi}_\mathcal{M} f(x) = & -\sum_{j \in \mathbf{N}} f_j \phi_j(x) \Psi(\lambda_j)\\
= & -\sum_{j \in \mathbf{N}} \left( \int_\mathcal{M} f(y) \phi_j(y) \mu(dy) \right) \phi_j(x) \Psi(\lambda_j)\\
= & -\int_\mathcal{M} f(y) \left( \sum_{j \in \mathbf{N}}  \Psi(\lambda_j) \phi_j(x) \phi_j(y) \right)  \mu(dy)\\
= & -\int_\mathcal{M} f(y) J(x,y) \mu(dy).
\end{align*}

\end{rk}
%
%
%

 We next discuss the Cauchy problems for the space fractional operators.
Recall that L\'{e}vy and Khintchine  showed that a L\'{e}vy measure $\nu$ on $\mathbb{R} - \{0\}$ can be $\sigma$-finite provided  that (\cite{appl2009})
\begin{equation*}
\int (|y|^2 \wedge 1) \nu(dy) < \infty.
\end{equation*}
Since, $(|y|^2 \wedge \epsilon) \leq (|y|^2 \wedge 1)$ whenever $\epsilon \in (0,1]$, it follows that
\begin{equation*}
\nu((-\epsilon, \epsilon)^c) < \infty, \quad \textrm{ for all }\quad 0 < \epsilon \leq 1
\end{equation*}
(see, for example, \cite{appl2009}). Furthermore, we recall that (\cite{Hoh, HohJacob})
\begin{equation}
|\Psi(\xi)| \leq c_\Psi\, (1+ |\xi|^2)
\end{equation}
where $c_\Psi = 2 \sup_{|\xi| \leq 1} |\Psi(\xi)|$, that is $\Psi \hat{f} \in L^2$ where $\hat{f}$ is the Fourier transform of $f$. For subordinators, similar calculation leads to
\begin{equation}
\epsilon^{-2} \int_0^\infty (y \wedge \epsilon )\nu(dy) = \int_0^\infty (y \wedge 1) \nu(dy) < \infty
\end{equation}
only if $\epsilon >0$ and $\epsilon \leq 1$. Therefore, we can write $\int_\epsilon^\infty \nu(dy) < \infty$ for $\epsilon >0$. Furthermore, for a subordinator $S_t$ with symbol $\Psi$ we also have that
\begin{equation}
\lim_{\xi \to \infty} \frac{\Psi(\xi)}{\xi} = 0 \label{zero-drift}
\end{equation}
that is, $S_t$ has zero drift and, $|\Psi(\xi)| < \xi$.

\begin{defin}
Let $\Psi$ be the symbol of a subordinator with no drift. Let $f \in H^s(\mathcal{M})$ and $s> (3n+3)/4$. Then,
\begin{equation}
\mathbb{D}^\Psi_\mathcal{M} f(m) =   -\sum_{j \in \mathbf{N}} f_j \phi_j(m) \Psi(\lambda_j). \label{def-der-gen-Psi}
\end{equation}
is absolutely and uniformly convergent. Furthermore,
\begin{align*}
\mathbb{D}^\Psi_\mathcal{M} f(m) = & -\int_\mathcal{M} f(y) J(m,y) \mu(dy)
\end{align*}
when the integral exists.
\end{defin}

 The definition above and therefore the convergence of \eqref{def-der-gen-Psi} immediately follows from \eqref{zero-drift} and the fact that $\|\phi_j \|_\infty \leq C \lambda_j^{(n-1)/4}$ for some $C>0$ with $\lambda_j \sim j^{2/n}$ as $j \to \infty$. 

%
%

\begin{rk}
Let us consider the kernel of the subordinate Brownian motion $B^x(S(t))$
\begin{equation}
q(x,y,t) = \int_0^\infty p(x,y,s) \, \P\{ S_t \in ds \} = \sum_{j \in \mathbf{N}} e^{-t \Psi(\lambda_j)} \phi_j(x)\phi_j(y)
\end{equation}
where $p(x,y,s)$ is the kernel \eqref{kernel-closed-e} which is the transition density of Brownian motion and $S$ is a subordinator with
\begin{equation}
-\partial_t\, \mathbb{E} e^{-\xi S_t} \Big|_{t=0^+} = \Psi(\xi).
\end{equation}
We observe that
\begin{equation}
- \partial_t\, q(x,y,t) \Big|_{t=0^+} = J(x,y).
\end{equation}

In this case the subordinate semigroup is given by
\begin{equation}\label{subordinate-semigroup}
P^\Psi_t f(x) = \int_0^\infty P_s f(x) \, \P\{ S_t \in ds \} = \sum_{j \in \mathbf{N}} e^{-t \Psi(\lambda_j)}<f, \phi_j>_\mu \phi_j(x)
\end{equation}
where $P_s$ is the transition semigroup of Brownian motion given in equation \eqref{semi-group-P}
\end{rk}
\begin{tm}
The solution to equation \eqref{space-frac-cauchy-problem} can be written as $u(x,t)=\mathbb{E}f(B^x_{S_t})=P^\Psi_t f(x)$ where $B^x_{S_t}$, $t\geq 0$ is a subordinate Brownian motion on $\mathcal{M}$ and $S_t=S^\Psi_t$ is a subordinator with Laplace exponent \eqref{symb-subordinator}.
\end{tm}
This is the so called Bochner subordination of Brownian motion with a subordinator $S_t$ which implies that $B^x_{S_t}$ is also a L\'evy process on $\mathcal{M}$. See more on the Bochner subordination in \cite{book-song-et-al}.

\section{Random fields on $\mathcal{M}$}
Let us consider the Gaussian random field $T(m)$, $m \in \mathcal{M}$ where $\mathcal{M}$ is a compact manifold with the following properties:
\begin{itemize}
\item [A.1)] $T$ has almost surely continuous sample paths;
\item [A.2)] $T$ has zero mean, $\mathbb{E}T(m) = 0$;
\item [A.3)] $T$ has finite mean square integral,
\begin{equation}
\mathbb{E}\left[ \int_\mathcal{M} T^2(m)\mu(dm) \right]  < \infty;
\end{equation}
\item [A.4)] $T$ has continuous covariance function
\begin{equation}
\mathscr{K}(m_1,m_2) = \mathbb{E}T(m_1)T(m_2).
\end{equation}
\end{itemize}
It is well-known (see for example \cite{KacSie47}) that there exist constants $\zeta_1 \geq \zeta_2 \geq \cdots \geq 0$ and continuous functions $\{\psi_j\}_{j \in \mathbf{N}}$ on $\mathcal{M}$ such that the following properties are fulfilled:
\begin{itemize}
\item [B.1)] $\{\psi_j\}_{j \in \mathbf{N}}$ are orthonormal in $L^2(\mathcal{M}, \mu)$;
\item [B.2)] $\{\psi_j, \zeta_j\}_{j \in \mathbf{N}}$ form a complete set of solutions to the Fredholm-type equation
\begin{equation}
\int_\mathcal{M} \mathscr{K}(m_1, m_2)\psi_j(m_1)\mu(dm_1) = \zeta_j\, \psi_j(m_2), \; \forall\, j \in \mathbf{N}; \label{fredholm-eq}
\end{equation}
\item [B.3)] the following holds true
\begin{equation}
\mathscr{K}(m_1, m_2) = \sum_{j \in \mathbf{N}} \zeta_j \psi_j(m_1) \psi_j(m_2) \label{RKHS-K}
\end{equation}
and the series is absolutely and uniformly convergent on $\mathcal{M} \times \mathcal{M}$;
\item [B.4)] there exists a sequence $\{\omega_j \}_{j \in \mathbf{N}}$ of Gaussian random variables ($\omega_j \sim N(0,1), \, \forall j$) such that the following Karhunen-Loeve expansion holds
\begin{equation}
T(m) = \sum_{j \in \mathbf{N}} \sqrt{\zeta_j} \omega_j \psi_j(m) \label{Kar-loe}
\end{equation}
and the series converges in the integrated mean square sense on $\mathcal{M}$.
\end{itemize}
The reader can consult the book by Adler \cite{Adler90}.

%
%
%
%
%

We introduce the following spectral representation for the random field $T$.
\begin{tm}
Let $\{\phi_j\}_{j \in \mathbf{N}}$ and $\{\psi_j\}_{j \in \mathbf{N}}$ be the orthonormal systems previously specified. Let $T(m)$, $m \in \mathcal{M}$ be the random field for which A.1-A.4 are fulfilled.
\begin{enumerate}
\item The representation
\begin{equation}
T(m) = \sum_{j \in \mathbf{N}} \phi_j(m) \, c_{j}  \label{Trep}
\end{equation}
where
\begin{align*}
c_j= & \int_\mathcal{M} T(m) \phi_j(m)\mu(dm), \quad j \in \mathbf{N}
\end{align*}
holds in $L^2(dP \otimes d\mu)$ sense, i.e.
\begin{equation}
\lim_{N \to \infty } \mathbb{E}\left[ \int_\mathcal{M} \left( T(m) -  \sum_{j =1}^N \phi_j(m) \, c_{j} \right)^2 \mu(dm) \right] = 0.
\end{equation}
\item The Fourier random coefficients $\{c_j\}_{j \in \mathbf{N}}$ are Gaussian r.v.'s written as
\begin{align*}
c_j = & \sum_{i \in \mathbf{N}} \sqrt{\zeta_i} \omega_i \langle \phi_j, \psi_i \rangle_{\mu}, \quad j \in \mathbf{N}
\end{align*}
where $\omega_j \sim N(0,1)$, $\forall j$ and therefore
\begin{equation}
c_j \sim N\left(0, \sum_{i \in \mathbf{N}} \zeta_i |\langle \phi_j, \psi_i\rangle_{\mu}|^2 \right), \quad j \in \mathbf{N}.
\end{equation}
Furthermore,
\begin{equation}
\mathbb{E}[c_k\, c_s] = \sum_{i \in \mathbf{N}} \zeta_i\, \langle \phi_k, \psi_i\rangle_\mu \, \langle \phi_s, \psi_i \rangle_\mu .\label{cov-coeff-c}
\end{equation}
\end{enumerate}
\label{TheoExpansion}
\end{tm}
\begin{proof}
We consider the orthonormal system $\{\phi_j \}_{j \in \mathbf{N}}$ on $L^2(\mathcal{M}, \mu)$ and the fact that the Karhunen-Loeve expansion
\begin{equation}
T(m) = \sum_{j \in \mathbf{N}} \sqrt{\zeta_j} \omega_j \psi_j(m) \label{K-L-exp}
\end{equation}
holds true since A.1-A.4 are fulfilled (the series converges in the integrated mean square sense on $\mathcal{M}$). In force of these facts we can write
\begin{equation}
\psi_j(x) = \sum_{i \in \mathbf{N}} \theta_{ij} \phi_i(x), \quad \textrm{where}\quad \theta_{ij}=\langle \psi_j, \phi_i \rangle_\mu \label{theta-ij}
\end{equation}
and $\{\psi_j\}_{j \in \mathbf{N}}$ is a set of continuous functions on $\mathcal{M}$ satisfying B.1-B.4.
Therefore, we obtain that
\begin{equation}
T(m) = \sum_{j,i \in \mathbf{N}}  \sqrt{\zeta_j}\, \omega_j \, \theta_{ij} \, \phi_i(m) = \sum_{i \in \mathbf{N}} \left( \sum_{j \in \mathbf{N}} \sqrt{\zeta_j}\, \omega_j \, \theta_{ij}  \right) \phi_i(m). \label{T-in-terms-of-psi}
\end{equation}
By comparing \eqref{T-in-terms-of-psi} with the \eqref{Trep}, we  can immediately see that
\begin{equation}
c_i = \sum_{j \in \mathbf{N}} \sqrt{\zeta_j}\, \omega_j \, \theta_{ij} \label{c-proof}
\end{equation}
term by term and $c_i$ is the Fourier random coefficient in the series expansion involving the orthonormal system $\{\phi_i\}_{i \in \mathbf{N}}$. On the other hand, from \eqref{K-L-exp}, we have that
\begin{equation}
c_j = \int_\mathcal{M} T(m)\phi_j(m)\mu(dm) = \sum_{i \in \mathbf{N}} \sqrt{\zeta_i} \, \omega_i \int_\mathcal{M} \psi_i(m) \, \phi_j(m)\, \mu(dm)
\end{equation}
which coincides with \eqref{c-proof}. We know that $\omega_j \sim N(0,1)$ and therefore,
\begin{equation}
c_i \sim N\left(0, \sum_{j \in \mathbf{N}} \zeta_j\, \theta^2_{ij} \right).
\end{equation}
From \eqref{RKHS-K} and \eqref{Kar-loe} we can immediately verify that $\omega_j$ for all $j \in \mathbf{N}$ are independent random variables, thus we write $\mathbb{E}[\omega_j\, \omega_i] = \delta_i^j$ which is the Kronecker's delta symbol
\begin{equation}
\delta_{i}^{j} = \left\lbrace \begin{array}{l}
1, \quad i=j,\\
0, \quad i \neq j.
\end{array} \right.
\end{equation}
Result \eqref{cov-coeff-c} comes from the fact that
\begin{align*}
\mathbb{E}[c_k\, c_s] = & \sum_{i,j \in \mathbf{N}} \sqrt{\zeta_i}\, \sqrt{\zeta_j}\,\mathbb{E}[\omega_j\, \omega_i]\, \langle \phi_k, \psi_i\rangle_\mu \, \langle \phi_s, \psi_j \rangle_\mu\\
= & \sum_{i,j \in \mathbf{N}} \sqrt{\zeta_i}\, \sqrt{\zeta_j}\,\delta_i^j\, \langle \phi_k, \psi_i\rangle_\mu \, \langle \phi_s, \psi_j \rangle_\mu
\end{align*}
and the claim appears.

From the completeness of the system $\{\phi_j \}_{j  \in \mathbf{N}}$ we also obtain that
\begin{align*}
\mathbb{E}\left[ \int_\mathcal{M} \left( T(m) -  \sum_{j =1}^N \phi_j(m) \, c_{j} \right)^2 \mu(dm) \right]
= & \mathbb{E}\left\| T(m) -  \sum_{j =1}^N \phi_j(m) \, c_{j} \right\|_\mu^2\\
= & \mathbb{E}\| T(m)\|^2_\mu - \sum_{j=0}^N \mathbb{E}c_j^2
\end{align*}
where
\begin{equation}
\mathbb{E}\| T(m)\|^2_\mu = \int_\mathcal{M} \mathbb{E}[T(m)]^2 \mu(dm)  = \sum_{j \in \mathbf{N}} \zeta_j
\end{equation}
and
\begin{equation}
\sum_{j=0}^N \mathbb{E}c_j^2 = \sum_{j=0}^N \sum_{i \in \mathbf{N}} \zeta_i \, |\langle \phi_j , \psi_i \rangle_\mu |^2. \label{c-c-proof}
\end{equation}
Indeed, we have that
\begin{equation}
\mathbb{E} [T(m)]^2= \mathscr{K}(m,m)
\end{equation}
and
\begin{equation}
\int_\mathcal{M} \mathscr{K}(m,m)\mu(dm) = \sum_{j \in \mathbf{N}} \zeta_j \int_\mathcal{M} |\psi_j(m)|^2\mu(dm) = \sum_{j \in \mathbf{N}} \zeta_j.
\end{equation}
Formula \eqref{c-c-proof} can be rewritten as
\begin{align*}
\sum_{j=0}^N \sum_{i \in \mathbf{N}} \zeta_i \, |\langle \phi_j , \psi_i \rangle_\mu |^2= & \sum_{j=0}^N \sum_{i \in \mathbf{N}} \zeta_i \, \langle \phi_j , \psi_i \rangle_\mu\, \langle \phi_j , \psi_i \rangle_\mu\\
= &  \bigg\langle \bigg\langle \sum_{i \in \mathbf{N}} \zeta_i \,   \psi_i(u) \, \psi_i(z), \sum_{j=0}^N \phi_j(u)\phi_j(z) \bigg\rangle_{\mu(du)} \bigg\rangle_{\mu(dz)}
\end{align*}
where, as usual,
\begin{equation*}
\langle f,g \rangle_\mu = \int_\mathcal{M} f(y)g(y)\mu(dy).
\end{equation*}
By observing that
\begin{equation}
\lim_{N \to \infty} \sum_{j=0}^N \phi_j(u)\phi_j(z) = \delta(u-z),
\end{equation}
by the completeness of $\{\phi_j\}_{j \in \mathbf{N}}$, we get that
\begin{align*}
\lim_{N \to \infty}\sum_{j=0}^N \sum_{i \in \mathbf{N}} \zeta_i \, |\langle \phi_j , \psi_i \rangle_\mu |^2= &\sum_{i \in \mathbf{N}} \zeta_i \, \| \psi_i \|_\mu = \sum_{i \in \mathbf{N}} \zeta_i.
\end{align*}
By collecting all pieces together we obtain that
\begin{align*}
\lim_{N \to \infty} \mathbb{E}\left[ \int_\mathcal{M} \left( T(m) -  \sum_{j =1}^N \phi_j(m) \, c_{j} \right)^2 \mu(dm) \right] =0
\end{align*}
and this concludes the proof.
\end{proof}

\begin{rk}
We observe that, if $\phi_j=\psi_j$ for all $j$ that is,  $\{ \phi_j \}$  is an orthonormal system of eigenfunctions with eigenvalues $\lambda_j$, $j \geq 0$ and solves the Fredholm-type equation \eqref{fredholm-eq} depending on $\zeta_j$, $j\geq 0$, then we have that
$$c_j = \sqrt{\zeta_j}\, \omega_j \sim N(0, \zeta_j).$$
\end{rk}

\subsection{Cauchy problems with random initial conditions}

We recall that the random field $T \in L^2(\mathcal{M})$ on the manifold $\mathcal{M}$ can be written as
\begin{equation}
T(m) = \sum_{j \in \mathbf{N}} \phi_j(m)\, c_j \label{T-rep-2}
\end{equation}
where the Fourier random coefficients are given in Theorem \ref{TheoExpansion}.
\begin{tm}
\label{theo-rf-pde}
The solution to
\begin{equation}
\left( \partial_t  - \mathbb{D}^\Psi_{\mathcal{M}} \right) u(m,t) = 0 \label{pde-rf}
\end{equation}
subject to the random initial condition
\begin{equation}
u(m,0)= T_0(m)=T(m) \in L^2(\mathcal{M}) \label{condition-stoch-pde}
\end{equation}
is the time dependent random field on $\mathcal{M}$ written as
\begin{equation}
u(m,t)= T^\Psi_t(m) = \sum_{j \in \mathbf{N}} e^{- t \Psi(\lambda_j )}  \phi_j(m)\, c_j . \label{T-rf-solution-pde}
\end{equation}
\end{tm}

\begin{proof}
By comparing \eqref{T-rep-2} with \eqref{T-rf-solution-pde} we immediately see that \eqref{condition-stoch-pde} is verified. Let us consider the fractional operator
\begin{align*}
\mathbb{D}^{\Psi}_\mathcal{M} u(m,t) = &- \int_\mathcal{M} u(y,t)J(m,y) \mu(dy)
\end{align*}
where the following expansion holds ($u \in L^2(\mathcal{M})$)
\begin{equation*}
u(x,t) = \sum_{j \in \mathbf{N}} e^{- t \Psi(\lambda_j )}  \phi_j(m)\, c_j, \quad x\in \mathcal{M}, \; t>0.
\end{equation*}
We can write
\begin{align}
\mathbb{D}^{\Psi}_\mathcal{M} u(m,t) = & -\sum_{j \in \mathbf{N}} e^{- t \Psi(\lambda_j )} c_j \int_\mathcal{M} \phi_j(y)  J(m,y) \mu(dy)\notag \\
= & \sum_{j \in \mathbf{N}} e^{- t \Psi(\lambda_j )} c_j \, \mathbb{D}^{\Psi}_\mathcal{M} \phi_j(m) \label{D-proof-stoch-pde}
\end{align}
where (see formula \eqref{D-op-one})
\begin{align*}
\mathbb{D}^{\Psi}_\mathcal{M} \phi_j(m) = & \int_0^\infty \left( P_s\, \phi_j(m) - \phi_j(m)\right) \nu(ds)\\
= & \int_0^\infty \left( \mathbb{E}\phi_j( B^m_s) - \phi_j(m)\right) \nu(ds)
\end{align*}
Since
\begin{align*}
\mathbb{E}\phi_j(B^m_s) = \sum_{i \in \mathbf{N}} e^{-s \lambda_i} \kappa_i \, \phi_i(m)\, \langle \phi_j, \phi_i \rangle_\mu
\end{align*}
where $\langle \phi_j, \phi_i \rangle_\mu= \delta_i^j$ and
\begin{equation}
\kappa_i = \int_\mathcal{M} \delta(y) \phi_i(y)\mu(dy) = 1.
\end{equation}  we get (see formula \eqref{law-B-on-M} as well)
\begin{equation}
P_s\, \phi_j(m) = e^{-s\, \lambda_j} \phi_j(m). \label{P-phi}
\end{equation}

By collecting all pieces together, we get
\begin{align*}
\mathbb{D}^{\Psi}_\mathcal{M} \phi_j(m) = & \int_0^\infty \left( e^{-s\, \lambda_j} \phi_j(m) - \phi_j(m)\right) \nu(ds)\\
= & \phi_j(m) \int_0^\infty \left( e^{-s\, \lambda_j} - 1\right) \nu(ds)\\
= & -\phi_j(m) \, \Psi(\lambda_j)
\end{align*}
where we have used the representation \eqref{symb-subordinator} of the symbol $\Psi$. In light of this, formula \eqref{D-proof-stoch-pde} takes the form
\begin{align*}
\mathbb{D}^{\Psi}_\mathcal{M} u(m,t) = & - \sum_{j \in \mathbf{N}} e^{- t \Psi(\lambda_j )} c_j \, \phi_j(m) \, \Psi(\lambda_j)
\end{align*}
and therefore,
\begin{equation*}
\left( \partial_t  - \mathbb{D}^\Psi_{\mathcal{M}} \right) u(m,t) = 0.
\end{equation*}

%
%

We also notice that
\begin{equation}
\sum_{j \in \mathbf{N}} \Big\| \Psi (\lambda_j) e^{-t \Psi(\lambda_j)}\, c_j \, \phi_j \Big\|_{\infty} \leq \sum_{j \in \mathbf{N}}  \Psi (\lambda_j) e^{-t \Psi(\lambda_j)}\, |c_j| \, \| \phi_j \|_{\infty} < \infty
\end{equation}
since $\Psi$ is a Bernstein function.
\end{proof}

\subsection{Special manifolds}
We give some examples of manifolds in the  following sections.

\subsubsection{The manifold $\mathcal{M} \equiv \mathbf{R}^n$.}

For the special case $\Psi(z)=|z|^\alpha$ we obtain
the fractional Laplacian
\begin{align}
\mathbb{D}^\Psi_{\mathbf{R}^n} f(x) =  & C_d(\alpha) \, \textrm{p.v.} \int_{\mathbf{R}^n} \frac{f(y) -f(x)}{|x-y|^{\alpha +d}}dy= -(-\triangle_{\mathbf{R}^d})^{\alpha} f(x), \quad x \in \mathbf{R}^d  \label{lap-princ-val}
\end{align}
where ''p.v.'' stands for the ''principal value'' being the integral above singular near the origin and $C_d(\alpha)$ is a normalizing constant depending on $d$ and $\alpha$.

\subsubsection{The manifold $\mathcal{M} \equiv \mathbf{S}^2$.} We consider the unit (two dimensional) sphere
\begin{align*}
\mathbf{S}^2 = & \left\lbrace z \in \mathbb{R}^3\, :\, |z| =1 \right\rbrace\\
= & \left\lbrace z \in \mathbb{R}^3\, :\, z =(\sin \vartheta \cos \varphi, \sin \vartheta \sin \varphi, \cos \vartheta), \, \vartheta \in [0, \pi], \, \varphi \in [0, 2\pi] \right\rbrace
\end{align*}
with
$$\mu(dz) = \sin \vartheta\, d\vartheta \, d \varphi.$$
The sphere $\mathbf{S}^2$ is an example of a compact manifold without boundary. For $\lambda_l= l(l+1)$ and $l \geq 0$, the spherical harmonics
\begin{equation*}
Y_{lm}(\vartheta, \varphi) = \sqrt{\frac{2l+1}{4 \pi} \frac{(l-m)!}{(l+m)!}} Q_{lm}(\cos \vartheta) e^{im\varphi}
\end{equation*}
solve the eigenvalue problem
\begin{equation}
\triangle_{\mathbf{S}_{1}^2} Y_{lm}= - \lambda_l \, Y_{lm} \label{eigenY}, \quad l \geq 0, \; |m| \leq l
\end{equation}
where
\begin{align}
\triangle_{\mathbf{S}_{1}^2} = &  \frac{1}{\sin \vartheta} \frac{\partial}{\partial \vartheta} \left( \sin \vartheta \frac{\partial}{\partial \vartheta} \right) + \frac{1}{\sin^2 \vartheta} \frac{\partial^2}{\partial \varphi^2}, \quad \vartheta \in [0,\pi],\; \varphi \in [0, 2\pi], \label{spherical-laplace}
\end{align}
is the spherical Laplace operator and
\begin{equation*}
Q_{lm}(z)=(-1)^m (1-z^2)^{m/2}\frac{d^m}{d z^m}Q_{l}(z)
\end{equation*}
are the associated Legendre functions with Legendre polynomials defined as
$$ Q_l(z) = \frac{1}{2^l l!}\frac{d^l}{dz^l} (z^2 - 1)^l. $$
For a detailed discussion see, for example, \cite{MarPeccBook}. We have that
\begin{equation}
\E(T(x)T(y))=\mathscr{K}(x,y) = \sum_{l \in \mathbf{N}} C_l \frac{2l+1}{4\pi} Q_l(\langle x, y \rangle) \label{cov-spherical}
\end{equation}
where
\begin{equation}
\langle x, y\rangle = \cos d(x,y)
\end{equation}
is the usual inner product in $\mathbf{R}^3$ ($d(x,y)$ is the spherical distance) and
\begin{equation}
\int_{\mathbf{S}^2_1} \mathscr{K}(x,y) Y_{lm}(y) \mu(dy) = C_l\, Y_{lm}(x).
\end{equation}
We recall that (addition formula)
\begin{equation*}
\sum_{m=-l}^{+l} Y_{lm}(x)Y^*_{lm}(y) = \frac{2l+1}{4\pi} Q_l(\langle x, y \rangle)
\end{equation*}
and we recover \eqref{cov-spherical} from \eqref{RKHS-K}. Thus, $\zeta_l = C_l$ (and $\lambda_l=l(l+1)$ as pointed out before) for all $l\geq 0$, which is the angular power spectrum of the spherical random field $T(x)$, $x\in \mathbf{S}^2$. Furthermore,
\begin{equation}
T(x) = \sum_{l \in \mathbf{N}} \sum_{|m| \leq l} c_{lm} Y_{lm}(x)
\end{equation}
where $c_{lm}$ are Gaussian r.v.'s (since  $T$  is Gaussian). Therefore, due to the fact that $\{Y_{lm}\}$ represents an orthonormal system of eigenfunctions of $\triangle_{\mathbf{S}^2}$ solving the Fredholm-type integral equation \eqref{fredholm-eq}, then $\theta_{ij}=\delta_i^j$ in \eqref{theta-ij} and we get that
$$c_{lm} = \sqrt{C_l} \omega_{l} \sim N(0, C_l), \quad  \textrm{for all }\, |m|\leq l \, \textrm{ and } \, l\geq 0.$$
Also we observe that
\begin{equation}
q(x,y,t) = \sum_{l \in \mathbf{N}} \sum_{m=-l}^{+l} e^{-t \Psi(\lambda_l)} Y_{lm}(x)Y^*_{lm}(y) = \sum_{l \in \mathbf{N}} e^{-t \Psi(\lambda_l)} \frac{2l+1}{4\pi} Q_{l}(\langle x, y\rangle)
\end{equation}
is the kernel solving the fractional equation
\begin{equation}
\partial_t u = \mathbb{D}^\Psi_{\mathbf{S}^2} u
\end{equation}
where
\begin{equation}
- \partial_t\, q(x,y,t) \Big|_{t=0^+} = J(x,y).
\end{equation}
For $\Psi(z)=z^\alpha$ we get that
\begin{equation}
\mathbb{D}^\Psi_{\mathbf{S}^2} u = -(-\triangle_{\mathbf{S}^2})^\alpha u.
\end{equation}
It is worth to mention that spherical random fields have been considered by many authors in order to study Cosmic Microwave Background (CMB) radiation in the theory explained by the Big Bang model. In particular, CMB radiation is a radiation filling the universe almost everywhere and it can be affected by several anisotropies. Our aim in this direction is to explain such anisotropies by considering our coordinates changed random field on $\mathcal{M}=\mathbf{S}^2_1$. This fact should become  clear later on, in Remark \ref{lastRk}. Here, the characterization of the angular power spectrum turns out to be very important. Indeed, such a study allow to explain many aspects such as Sachs-Wolfe effect or Silk damping effect for instance. For a deep discussion on this topic we refer to \cite{MarPeccBook} or our recent paper \cite{DovNane1}.


\subsubsection{The manifold  $\mathcal{M} \equiv \mathbf{T}^2$: Torus}
 We now consider the compact (two dimensional) manifold $\mathbf{T}^2$ which is the quotient of the unit square $Q=[0,1]^2 \subset \mathbf{R}^2$ by the equivalence relation
\begin{equation*}
(x,y) \sim (x+1, y) \sim (x, y+1 )
\end{equation*}
equipped with the quotient topology. It is well known that the Laplace-Beltrami operator on the $n$-torus is written as $\triangle_{\mathbf{T}^n} = \sum_{j=1}^n \frac{\partial^2}{\partial x_j^2}$ where $x_j$ is a variable such that, for all $j$, it describes the circle $\mathbf{S}^1 =\{e^{ix_j}:\, -\pi < x_j < \pi\}$. An integrable function $f$ on $\mathbf{T}^n$ is therefore written as
\begin{equation*}
f(x) = \sum_{k \in \mathbf{Z}^n} f_k\, e^{i(k \cdot x)}, \quad x \in \mathbf{T}^n
\end{equation*}
where $k \cdot x = \sum_j k_j x_j$ and
\begin{equation*}
f_k = \frac{1}{(2\pi)^n}\int_{\mathbf{T}^n} f(x) e^{-i (k \cdot x)}\mu(dx).
\end{equation*}
The heat kernel on $\mathbf{T}^n$ takes the form
\begin{equation}
p(x,t) = \frac{1}{(4\pi)^{n/2}} \sum_{k \in \mathbf{Z}^n} \exp \left(- \frac{|x-2\pi k|^2}{4t} \right)
\end{equation}
and the transition semigroup is therefore written as
\begin{equation}
P_tf(x) = \int_{\mathbf{T}^n} p(x-y,t) f(y)\mu(dy).
\end{equation}
Also in this case we can study the differential operator $-\Psi(-\triangle)$. See for example the paper by Bochner \cite{Bochner-torus} or the recent work \cite{roncal-stinga} where the authors investigate the fractional power of the Laplace operator on the torus.

%
%

\section{ Time changed Brownian manifolds }

In this section we  introduce the time-dependent random field
\begin{equation}
\mathfrak{T}_t(m)  = T(B^m_t),\quad m\in \mathcal{M}, \; t>0
\end{equation}
which can be conveniently written as
\begin{equation}
\mathfrak{T}_t(m) = \sum_{j \in \mathbf{N}} \phi_j( B^m_t)\, c_j
\end{equation}
where $B^m_t$, $t\geq 0$ is a Brownian motion on $\mathcal{M}$ started at $m \in \mathcal{M}$ and $T$ is the random field on $\mathcal{M}$ with representation \eqref{Trep}. In this section we study the compositions involving both the subordinate and the time changed  Brownian motions leading to the time-dependent random fields
\begin{equation}
\mathfrak{T}^\Psi_t(m) = T( B^m_{S_t}),\quad m\in \mathcal{M}, \; t>0, \; \alpha \in (0,1) \label{T-rf-subordinateBM}
\end{equation}
and
\begin{equation}
\mathfrak{T}^\beta_t(m) = T( B^m_{E_t}),\quad m\in \mathcal{M}, \; t>0, \; \beta \in (0,1). \label{T-rf-timechangedBM}
\end{equation}
Recall that $D_t$ is a subordinator with symbol $\Psi$ whereas $E_t=E_t^\beta$ is an inverse to a stable subordinator of index $\beta \in (0,1)$ defined by the formula \eqref{inverse-stable}. We assume that the random times $S_t$ and $E_t$ are independent from $B^m_t$. In the following results we obtain coordinate changed random fields starting from the random fields  \eqref{T-rf-subordinateBM} and \eqref{T-rf-timechangedBM} indexed by different time-changed Brownian manifolds. We consider the sets $\{ B^m_{S_t} \} \subset \mathcal{M}$ and $\{ B^m_{E_t}\} \subset \mathcal{M}$ as new sets of indices for the random field $T$.

\subsection{Space-time fractional equations with random initial conditions} In this section we relate the solutions 
to the equations \eqref{pde-rf} and \eqref{pdeFracmanifold} (with random initial condition) with the coordinate 
changed random fields introduced so far. 
\begin{lem}\label{lemma1}
The random field in  \eqref{T-rf-solution-pde} can be represented as
\begin{equation}
T^\Psi_t(m) = \mathbb{E}\left[ T( B^m_{S_t}) \Big| \mathfrak{F}_T \right] \label{T-alpha-lemma}
\end{equation}
where $\mathfrak{F}_T$ is the $\sigma$-field generated by the random field $T$.
\end{lem}
\begin{proof}
First we write (see \eqref{Trep})
\begin{align*}
T( B^m_{S_t}) = \sum_{j \in \mathbf{N}} \phi_j(B^m_{S_t})\, c_j.
\end{align*}
Thus,
\begin{align*}
\mathbb{E}\left[ T( B^m_{S_t}) \Big| \mathfrak{F}_T \right] = & \sum_{j \in \mathbf{N}} \mathbb{E}\left[ \phi_j( B^m_{S_t})\, c_j \Big| \mathfrak{F}_T \right]
\end{align*}
where, we recall that
\begin{equation*}
c_j = \int_\mathcal{M} T(m) \phi_j(m)\mu(dm)
\end{equation*}
and therefore
\begin{align*}
\mathbb{E}\left[ T( B^m_{S_t}) \Big| \mathfrak{F}_T \right] = & \sum_{j \in \mathbf{N}} \mathbb{E}\left[ \phi_j( B^m_{S_t}) \right] \, c_j .
\end{align*}
We have that
\begin{equation*}
\mathbb{E}\left[ \phi_j( B^m_{S_t}) \right] = \mathbb{P}_t \, \phi_j(m)
\end{equation*}
where $\mathbb{P}_t = \exp(t \mathbb{D}^\Psi_\mathcal{M})$ is the semigroup associated with the problem in Theorem \ref{theo-rf-pde}. As in formula \eqref{P-phi} we get that
\begin{equation}
\mathbb{P}_t \, \phi_j(m) = e^{-t \Psi(\lambda_j)} \phi_j(m)
\end{equation}
and
\begin{align*}
\mathbb{E}\left[ T( B^m_{S_t}) \Big| \mathfrak{F}_T \right] = & \sum_{j \in \mathbf{N}} e^{-t \Psi(\lambda_j)} \phi_j(m) \, c_j
\end{align*}
which is the spectral representation of $T^\alpha_t$.
\end{proof}

Next we introduce the space of random fields  on $\mathcal{M}$ given by
\begin{equation}
\mathbb{H}_F^s(\mathcal{M}) = \left\lbrace T \, \textrm{ such that \eqref{Trep} holds and }\, \sum_{j \in \mathbf{N}} (\lambda_j)^{2s} \, \mathbb{E}c_j^2 < \infty  \right\rbrace
\end{equation}
that is $\mathbb{H}_F^s(\mathcal{M}) \subset L^2(\mathcal{M})$.
\begin{tm}
For $s > (3+3n)/4$, the solution to the problem \eqref{pdeFracmanifold} with random initial condition $T_0^\beta(m) = T(m) \in \mathbb{H}_F^s(\mathcal{M})$, $m \in \mathcal{M}$ is written as
\begin{equation}
T^\beta_t(m) = \sum_{j \in \mathbf{N}} E_{\beta}(-t^\beta \lambda_j) \, \phi_j(m)\, c_j, \quad  m \in \mathcal{M},\,  t \geq 0 \label{T-rf-solution-pde-time-frac}
\end{equation}
where
\begin{equation*}
c_j = \int_{\mathcal{M}} T(x) \phi_j(x) \mu(dx), \quad j \in \mathbf{N}
\end{equation*}
and \eqref{T-rf-solution-pde-time-frac} holds in $L^2(dP \otimes d\mu)$ sense, i.e.
\begin{equation}
\lim_{L \to \infty} \E \left[ \int_\mathcal{M} \left( T^\beta_t(m) - \sum_{j =1}^L E_{\beta}(-t^\beta \lambda_j) \, \phi_j(m)\, c_j \right)^2 \mu(dm) \right]=0.
\end{equation}
\end{tm}
\begin{proof}
We use the same arguments as in the proof of Theorem \ref{Theorem-closed-fractional} and Theorem  \ref{TheoExpansion}. The initial condition is satisfied by taking into account that $E_\beta(0)=1$, see formula \eqref{mittag-leffler-function}.
\end{proof}

\begin{lem}
The random field in Equation \eqref{T-rf-solution-pde-time-frac} can be represented as
\begin{equation}
T^\beta_t(m) = \mathbb{E}\left[ T( B^m_{E_t}) \Big| \mathfrak{F}_T \right] \label{T-beta-lemma}
\end{equation}
where $\mathfrak{F}_T$ is the $\sigma$-field generated by the random field $T$.
\end{lem}
\begin{proof}
As in  proof of Lemma \ref{lemma1} we can write
\begin{align*}
T( B^m_{E^\beta_t}) = \sum_{j \in \mathbf{N}} \phi_j( B^m_{E^\beta_t})\, c_j
\end{align*}
where we use also the superscript $\beta$ in order to underline the connection with Theorem \ref{Theorem-closed-fractional}. As before, we have that
\begin{equation*}
\mathbb{E}\left[ T( B^m_{E^\beta_t}) \Big| \mathfrak{F}_T \right] = \sum_{j \in \mathbf{N}} \mathbb{E}\left[\phi_j( B^m_{E^\beta_t})\right]\, c_j.
\end{equation*}
From \eqref{sol-closed-fractional} and the orthogonality of $\{\phi_j \}$, we have that
\begin{align*}
\mathbb{E}\left[\phi_j( B^m_{E^\beta_t})\right] = E_{\beta}(-t^\beta \lambda_j) \, \phi_j(m)
\end{align*}
and therefore,
\begin{equation*}
\mathbb{E}\left[ T( B^m_{E^\beta_t}) \Big| \mathfrak{F}_T \right] =  \sum_{j \in \mathbf{N}} E_{\beta}(-t^\beta \lambda_j) \, \phi_j(m)\, c_j
\end{equation*}
which coincides with the spectral representation \eqref{T-rf-solution-pde-time-frac}.
\end{proof}

\subsection{Spectrum for time-changed random fields}
We recall that, for the random Fourier coefficients we have that
\begin{align*}
\mathbb{E}[c_k c_j] = & \int_\mathcal{M}\int_\mathcal{M} \mathscr{K}(x,y) \phi_k(x)\phi_j(y) \mu(dx)\mu(dy) \\
= & \sum_{i \in \mathbf{N}} \zeta_i \int_\mathcal{M}\int_\mathcal{M} \psi_i(x)\psi_i(y) \phi_k(x)\phi_j(y) \mu(dx)\mu(dy)\\
= & \sum_{i \in \mathbf{N}} \zeta_i \, \theta_{ki}\theta_{ji}
\end{align*}
as pointed out before. For $k=j$, we get the spectrum
\begin{equation}
C_j = \mathbb{E}c_j^2 = \sum_{i \in \mathbf{N}} \zeta_i \, \theta_{ji}^2, \quad j=0,1,2,\ldots .
\end{equation}
From the fact that
\begin{align*}
\int_\mathcal{M}T^2(x)\mu(dx) = \sum_{j \in \mathbf{N}} c_j^2
\end{align*}
since $T\in L^2(\mathcal{M})$ we get that
\begin{align*}
\sum_{j \in \mathbf{N}} \mathbb{E} c_j^2 = \mathbb{E} \left( \int_\mathcal{M}T^2(x)\mu(dx) \right) < \infty
\end{align*}
and therefore in particular, if
\begin{equation}
\mathbb{E}c_j^2 = C_j \sim j^{-\gamma}, \quad \gamma >2 \label{spec-T}
\end{equation}
then this ensures summability and $T$ is a square integrable random field on $\mathcal{M}$.

We present the following results concerning the spectrum of the random fields introduced so far.
\begin{tm}
Let $c^\Psi_j(t)$, $t\geq 0$, $j \in \mathbb{N}$ be the spectrum of \eqref{T-alpha-lemma} and suppose that $c_j$'s satisfy \eqref{spec-T}. Then,
\begin{equation}
\mathbb{E} [c^\Psi_j(t)]^2 = C_j e^{-2t\Psi(\lambda_j)} \approx j^{-\gamma} e^{-t \Psi(j^{2/n})}, \qquad \textrm{ as } \, j \to \infty. \label{spec-alpha-asymptotocs}
\end{equation}
\end{tm}
\begin{proof}
We obtain that
\begin{align*}
\int_{\mathcal{M}} T^\Psi_t(m) \phi_j(m)\mu(dm) = & \sum_{i \in \mathbf{N}} c_i e^{-t\Psi(\lambda_i)} \int_\mathcal{M} \phi_i(m) \phi_j(m)\mu(dm) \\
= & c_j e^{-t \Psi(\lambda_j)}
\end{align*}
where
$$c^\Psi_j(t) = c_j e^{-t \Psi(\lambda_j)}$$
is the spectrum of $T^\Psi_t$. From \eqref{eigenvalue-asymptotics} and \eqref{spec-T}, formula \eqref{spec-alpha-asymptotocs} immediately follows.
\end{proof}

\begin{tm}
Let $c^\beta_j(t)$, $t\geq 0$, $j \in \mathbb{N}$ be the spectrum of \eqref{T-beta-lemma} and suppose that $c_j$'s satisfy \eqref{spec-T}. Then,
\begin{equation}
\mathbb{E} [c^\beta_j(t)]^2 = C_j [E_\beta(-t^\beta \lambda_j)]^2 \approx j^{-\gamma} (1+t^\beta j^{2/n})^{-2}, \qquad \textrm{ as } \, j \to \infty. \label{spec-beta-asymptotics}
\end{equation}
\end{tm}
\begin{proof}
We have that
\begin{align*}
\int_{\mathcal{M}} T^\beta_t(m) \phi_j(m)\mu(dm) = & \sum_{i \in \mathbf{N}} c_i E_{\beta}(-t^\beta \lambda_i) \int_\mathcal{M} \phi_i(m) \phi_j(m)\mu(dm) \\
= & c_j  E_{\beta}(-t^\beta \lambda_j)\\
= & c_j^\beta(t).
\end{align*}
From \eqref{eigenvalue-asymptotics}, \eqref{spec-T} and \eqref{Ebound} we get \eqref{spec-beta-asymptotics}.
\end{proof}

From the fact that
\begin{equation}
L^2(\mathcal{M}) = \bigoplus_{j=1}^\infty \mathcal{H}_j
\end{equation}
where $\mathcal{H}_j$, $j=1,2,\ldots$ are orthogonal eigenspaces, we have that $c_j^\Psi$ and $c^\beta_j$ represent the variances of $T^\Psi$ and $T^\beta$ respectively, which are explained  by their projections to $\mathcal{H}_j$.
\begin{rk}
\label{lastRk}
We recall some symbols of the subordinators introduced before:
\begin{itemize}
\item $\Psi(z) = z^\alpha$: stable subordinator, $\nu(dy) = dy\, \alpha y^{-\alpha -1}/ \Gamma(1-\alpha)$;
\item $\Psi(z) = bz + z^\alpha$: stable subordinator with drift, $\nu(\cdot)$ as above and $b>0$;
\item $\Psi(z) = \ln (1 + z )$: gamma subordinator, $\nu(dy) =dy\, y^{-1}e^{-y}$;
\item $\Psi(z) = \ln (1+ z^\alpha)$: geometric stable subordinator, $\nu(dy)=dy\, \alpha y^{-1}E_\alpha(-y)$ where $E_\alpha$ is the Mittag-Leffler function.
\end{itemize}
If we consider the sum of an $\alpha_1$-stable subordinator $X$ and an $\alpha_2$-geometric stable subordinator $Y$, then we get that
\begin{align*}
\mathbb{E} \exp\left( -\lambda_j (X_{qt} + Y_{pt})\right) = & \exp\left( - qt\Psi_X(\lambda_j) - pt \Psi_Y(\lambda_j) \right) \\
= & e^{-qt \lambda_j^{\alpha_1} } (1+\lambda_j^{\alpha_2})^{-pt}
\end{align*}
and therefore, for the spectrum of \eqref{T-alpha-lemma} we obtain that
\begin{equation}
\mathbb{E}[c^\Psi_j(t)] \approx j^{- \gamma - \frac{2 p t}{n}\alpha_2 } \exp( -q t j^{\frac{2}{n}\alpha_1 })
\end{equation}
for large $j$. Obviously, for $p=0$ or $q=0$, we have qualitatively different behaviour for the covariance structure of $c_j^\alpha$ and therefore of the corresponding field $T^\alpha$.
\end{rk}
\begin{rk}
We notice that $c^\Psi_j(t)$ and $c^\beta_j(t)$ approach to zero as $t \to \infty$. Furthermore, the random solutions \eqref{T-alpha-lemma} and \eqref{T-beta-lemma} converge to the random variables $T^\Psi_t (m) \stackrel{t \to \infty}{\longrightarrow} \phi_0(m) c_0$ and $T^\beta_t (m) \stackrel{t \to \infty}{\longrightarrow}  \phi_0(m)c_0$ under the assumption that $\lambda_0=0$ and therefore $\Psi(\lambda_0)=0$ and $E_\beta(-\lambda_0) = 1$. Thus the steady state solution in both cases turns out to be a random variable on $\mathcal{M}$ with law given by  $c_0$ (see also Theorem \ref{thm-toharmonic}).
\end{rk}

\end{document}